\def\ZZ         {{\mathbb Z}}

\def\RR         {{\mathbb R}}
\def\CC         {{\mathbb C}}
\def\QQ         {{\mathbb Q}}
\def\PP         {{\mathbb P}}

\def\A          {{\cal A}}

\def\E          {{\cal E}}
\def\X          {{\cal X}}

\def\I        {{\cal I}}

\def\N          {{\cal N}}    
\def\O          {{\cal O}}

\def\deg        {{\rm deg}}

\def\cal        {\mathcal}

\documentclass{amsart}
\usepackage{amssymb}
\usepackage{verbatim}
\usepackage{eucal}
\usepackage{mathrsfs}
\usepackage{graphicx}
\usepackage{psfrag}

\newtheorem{theorem}{Theorem}[section]
\newtheorem{lemma}[theorem]{Lemma}
\newtheorem{prop}[theorem]{Proposition}
\newtheorem{corollary}[theorem]{Corollary}

\theoremstyle{definition}
\newtheorem{dfn}[theorem]{Definition}

\newtheorem{example}[theorem]{Example}

\theoremstyle{remark}
\newtheorem{remark}[theorem]{Remark}

\begin{document}

\title[Mordell-Weil groups of isotrivial abelian varieties]
{On Mordell-Weil groups of isotrivial abelian 
varieties over function fields}

\author{Anatoly Libgober}

\address{
Department of Mathematics\\
University of Illinois\\
Chicago, IL 60607}
\email{libgober@math.uic.edu}

\footnote{Supportd by  grant from Simons Foundation}

\begin{abstract} 
We show that the Mordell-Weil rank 
of an isotrivial abelian variety 
with cyclic holonomy  
depends only on the fundamental group 
of the complement to the discriminant,
provided the discriminant has 
singularities in CM class introduced here. 
This class of singularities 
includes all unibranched plane 
curves singularities. As a corollary, 
we describe a family of simple Jacobians over the field 
of rational functions in two variables 
for which the Mordell Weil 
rank is arbitrarily large.
\end{abstract}

\maketitle

\section{Introduction}

Let $\A$ be an abelian variety over a function field $K$
of characteristic zero.
The group of $K$-points of $\A$ is an interesting 
{\it algebro-geometric} 
invariant. If $dim_K \A=1$, $deg.tr.K/\CC=n$,  
then it is closely related to the Neron-Severi group 
of the corresponding elliptic $(n+1)$-fold (cf. \cite{survey}, \cite{wazir}).
In this note we consider a class of abelian varieties 
$\A$ over the function field $K=\CC(x,y)$
for which the Mordell-Weil rank can be described
in {\it topological} terms. This description extends
the results of \cite{jose} where the case of elliptic 
curves over $K=\CC(x,y)$ was studied in detail.

We shall work with a non-singular projective model
of $\A$, i.e. assume that $\A$ is a smooth projective variety 
together with a flat morphism 
\begin{equation}\label{morphismoffamily}
\pi: \A \rightarrow \PP^2
\end{equation}
such that fibers over closed points in the Zariski 
open subset of $X=\PP^2$ are polarized abelian varieties over $\CC$.
Our main results relate the Mordell-Weil rank of $\A$ 
to the fundamental group of the complement to the discriminant 
$\Delta$ of (\ref{morphismoffamily}).

The restrictions which we impose on 
the abelian variety $\A$, allowing one to express the Mordell-Weil rank 
topologically, are the following:

\noindent A1. $\A$ is isotrivial. 

\noindent A2. The holonomy group of the family (\ref{morphismoffamily}), 
(cf. section \ref{holonomysubsection}) is cyclic.

\noindent A3. The singularities of the discriminant have CM type 
(cf. section \ref{cmsingsection}).

\bigskip

In the case of elliptic {\it surfaces} (i.e.
$dim_K \A=1$, $deg.tr.K/\CC=1$) satisfying the condition A1,
the condition A2 is automatically fulfilled. However, the Mordell-Weil 
rank is far from 
being topological as the examples in \cite{factors} show.
Besides giving bounds on the Mordell-Weil rank (cf. Theorem \ref{main}),  
we also present several classes of families $\A$ 
for which the rank can be calculated explicitly. 
In addition to the conditions A1-A3 above,
we limit our-selves to the case of abelian varieties for which 
the discriminant is irreducible. We imposed this condition to simplify 
the exposition. 
 


The data which affects the Mordell-Weil rank of $\A$ 
in fact requires only a small portion of the fundamental group
of the complement to $\Delta$. It is 
the same as the data controling the Betti numbers of cyclic
branched covers of $\PP^2$ with the ramification locus
coinciding with the discriminant of morphism (\ref{morphismoffamily}).
As was shown in \cite{duke}, Betti numbers of cyclic branched 
covers can be expressed in terms of the quotient $\pi_1/\pi_1''$ 
of the fundamental group $\pi=\pi_1(\PP^2-\Delta)$ 
by its second commutator $\pi_1''$. It is convenient 
to express them   
in terms of the Alexander invariant of $\Delta \subset \PP^2$
i.e.  the vector space $\pi_1'/\pi_1'' \otimes \CC$ 
considered as a module over the group ring $\CC[\pi_1/\pi_1']$
of the abelianization of $\pi_1$ and 
ultimately this Alexander invariant represents the topological data which 
controls Mordell Weil ranks of abelian varieties (\ref{morphismoffamily}).

Results of this note show that for abelian varieties considered below, 
the Mordell-Weil rank depends, besides the type of the generic fiber,  
the degree and the local type of singularities of the discriminant, 
on the dimensions of certain linear systems 
of curves determined by the local type of singularities 
of the latter. This is a consequence of known results 
showing that the Alexander invariants of plane singular curves
depend only on this data (cf. \cite{lectures} and section \ref{sectionalexander}
below and references therein). 
Recently in \cite{remke}, a relation 
was obtained between the rank of the elliptic curves
and the dimensions of such linear systems 
in the case when $\A$ has the discriminant {\it with  
cusps and nodes as the only singularities}, 
using methods different than those used in this 
paper (i.e. studying the syzygies of the locus of cusps 
of the discriminant).

One of the key ingredients in the proof of above mentioned results, 
having independent interest, is the decomposition 
theorem of the Albanese varieties of cyclic 
branched covers of $\PP^2$. In the context of abelian 
varieties (\ref{morphismoffamily}) these covers come
up since the abelian varieties satisfying above conditions A1 and A2 
 become trivial over cyclic extensions
of $\CC(x,y)$.
We show that the Albanese variety of the cyclic cover of $\PP^2$, 
over which the pull back of $\A$ is trivial
splits up to isogeny into a product of abelian varieties of CM type, 
assuming that  
the singularities of the branching locus of the 
cyclic cover of $\PP^2$ have CM type. 
More precisely, we have the following (similar 
result was obtained in \cite{jose} in the case when 
$\A$ is an elliptic curve but with slightly different
assumptions on singularities):
\begin{theorem}\label{mainresultintro} 
Let $\Delta$ be an irreducible and reduced curve in $\PP^2$  
such that all its singularities have CM type.
Then the Albanese variety of a
cyclic cover of $\PP^2$ ramified along $\Delta$ is isogenous to 
a product of abelian varieties of CM type each having as its 
 endomorphism algebra an etale algebra which is 
a product of cyclotomic fields.
\end{theorem}
The definition of singularities of a CM type in the case of 
plane curves is given in terms
of the local Albanese variety which is equivalent to the data
of weight one part and its Hodge filtration for 
the mixed Hodge structure on the cohomology of the Milnor fiber of 
the singular points (cf. Definition \ref{singcmtypedefinition}).
The local Albanese variety is the special case of the abelian
variety associated by Deligne with 1-motif in \cite{deligneIII}. 
We say that a plane curve singularity has a 
CM type if its local Albanese variety 
has a CM type.
We refer to \cite{shimura}, \cite{mum} or \cite{milne} for information
on abelian varieties of CM type but recall that those are
abelian varieties $A$ with $End(A)\otimes \QQ$ containing 
an (etale) $\QQ$-subalgebra of rank 
$2 {\rm dim} A$ isomorphic to a product of fields 
(in the case of CM-singularities, we show that
these fields are cyclotomic).

The class of plane curve singularities of CM type is rather large: 
it includes all unibranched singularities 
(cf. theorem \ref{localalbanesemain}), simple singularities, 
$\delta$-essential singularities in the sense of \cite{jose} 
etc. 
However, ordinary multiple points of multiplicity greater
than 3 do not have CM type in general  (cf. Section \ref{localalbsection}).

The precise relation between the topology of the complement 
to the discriminant and the Mordell-Weil rank is given as follows.

\begin{theorem}\label{main} Let 
$\A$ be an isotrivial abelian variety over field $\CC(x,y)$, 
$\pi$ be morphism (\ref{morphismoffamily}) and $A$ be its generic fiber. 
Let $\Delta \subset \PP^2$
be the discriminant of $\pi$ and let $G \subset Aut A$ be the 
holonomy group of $\A$ (cf. \ref{holnomy}).
Assume that: 

a) $G$ is a cyclic group of order $d$ acting on generic fiber $A$ 
of (\ref{morphismoffamily})
without fixed subvarieties of a 
positive dimension.

b) The singularities of $\Delta$ have CM type 
and $\Delta$ is irreducible.

Then

1. the rank of the Mordell-Weil group of $\A$ is zero, 
unless the generic fiber of $\pi$ is an abelian variety of CM-type 
with endomorphism algebra containing a cyclotomic field.

2. Assume that the generic fiber $A$ of $\pi$ is a {\it simple} abelian variety 
of CM type corresponding to the field $\QQ(\zeta_{d})$.
Let $s$ be the multiplicity of the 
factor $\Phi_d(t)$ of the Alexander polynomial 
of $\pi_1(\PP^2-\Delta)$ where 
$\Phi_d(t)$ is the cyclotomic polynomial of degree $d$.
Then:
\begin{equation}\label{inequalitymain}
{\rm rk}MW(\A,\CC(x,y)) \le  s \cdot \phi(d)
\end{equation} 
(here $\phi(d)=deg\Phi_d(t)$ is the Euler function).

3. Let $A$ be an abelain variety as in 2.
If $d$ is the order of the holonomy of $\A$ and the Albanese 
variety $Alb(X_d)$ of the $d$-fold cover $X_d$ of $X$ 
ramified over $\Delta$ 
has $A$ is its direct summand 
with multiplicity $s$ then one has equality in (\ref{inequalitymain}).
\end{theorem}
Theorem \ref{main} has the following as an immediate consequence:
\begin{corollary}\label{coromain}
If $\A$ is a family (\ref{morphismoffamily}) with generic 
fiber $A$ such that $End(A) \otimes \QQ=\QQ(\zeta_d)$ holds and  
such that for each singular point of the discriminant 
the monodormy operator has no 
primitive roots of unity of degree $d$ as an eigenvalue then $rkMW(\A)=0$.

On the other hand, for the Jacobian of the curve
over $\CC(x,y)$ given 
in $(u,v)$ plane by the equation
\begin{equation}\label{finalexampleintro} 
u^p=v^2+(x^p+y^p)^2+(y^2+1)^p
\end{equation}
one has $rkMW={p-1}$ (cf. \ref{finalexample}). 
The Jacobian of generic fiber of the family
(\ref{finalexampleintro}) is a simple abelian variety.
\end{corollary}

It would be interesting to know if ranks of isotrivial
abelian varieties over $\CC(x,y)$ with fixed generic fiber 
are bounded (or have non-trivial bounds in terms of
degree of discriminant). 
Note that inequality (\ref{inequalitymain}) 
only shows that with assumption of Theorem \ref{main}
such bounds are equivalent to the 
bounds on the possible multiplicities  of $\Phi_d$ in the 
Alexander polynomial of the discriminant.
Very little is known at the moment about bounds on such 
multiplicites (cf. \cite{jose} containing a discussion of the relation between
the bounds on the rank and the degree of the latter).


The content of this paper is as follows. 
In the next section we recall the 
background material used below.
The section \ref{localalbsection} discusses the local Albanese varieties 
of plane curve singularities and cases when they have CM type.
The decomposability of the Albanese variety of cyclic branched covers 
(under certain conditions) is proved in section \ref{splittingsection}. 
Section \ref{proofmain} contains the proof of the theorem 
\ref{main}  
and gives 
examples of specific situations in which the above theorem 
can be applied. The Theorem \ref{main} in fact can be used in both 
directions: it gives many examples in which one obtains
explicitly the rank of the Mordell-Weil group. On the other hand,
it also provides a mean to give a bound on the complexity of the Alexander 
module of certain curves (cf. \cite{jose}). The example \ref{finalexample},
discussing the discriminant with largest known 
multiplcity of $\Phi_d(t)$ is given at the end of the last section.

Finally, I want to thank J.I.Cogolludo for comments on this paper.

\begin{section}{Abelian varieties 
over transcendental extensions of $\CC$}\label{intro}

\subsection{Isotrivial abelian varieties, discriminant and holonomy}
\label{holonomysubsection}

As in the Introduction, we fix a flat proper morphism 
of smooth complex projective varieties 
$\pi: \A \rightarrow X$ with generic fiber being an abelian variety 
over $\CC$, i.e. an abelian variety over $\CC(X)$.

A {\it rational} section (resp. a section) of $\pi$ 
is a rational (resp. regular) 
map $s: X \rightarrow \A$ such that 
$\pi \circ s$ is the identity on the domain of $s$.

An abelian variety  $\pi: \A \rightarrow X$ is called isotrivial if 
for an open set $U \subseteq X$ the fibers of $\pi$  
over any pair of points 
$x,y \in U$ are isomorphic as polarized abelian varieties
 with the polarization induced from $\A$.
The generic fiber of $\pi$ will be denoted $A$.

The discriminant locus $\Delta$ of $\pi$ is the subvariety 
of $X$ consisting of points $x$ for which the fiber $\pi^{-1}(x)$ 
is not smooth. The map $\pi^{-1}(X-\Delta) \rightarrow X-\Delta$ is a locally 
trivial fibration.
It follows from \cite{levin} (cf. also \cite{kollar} and \cite{serrano})
that there is an unramified Galois covering 
$s: X'-\Delta' \rightarrow X-\Delta$ such that

a) the Galois group $G$ is a subgroup of automorphisms of the fiber
$A$ preserving its polarization induced from the polarization of $\A$ and

b) such that 

\begin{equation}\label{levin}
   \A=\{(X'-\Delta') \times A\}/G
\end{equation}

\noindent
with the action given by $g(x,a)=(gx,ga) \ \ (x \in X'-\Delta', a \in A)$.
The equality (\ref{levin}) is a birational isomorphism which is biregular 
if one replaces the left hand side by the open subset $s^{-1}(X-\Delta)$
in $\A$.

We shall assume that  $X'-\Delta'$ is an open set in its $G$-equivariant 
smooth compactification $X'$, i.e $X'-\Delta'$ is 
the complement to a divisor $\Delta' \subset X'$
where $X'$ is a $G$-equivariant resolution of singularities of
a $G$-equivariant compactification of $X'-\Delta'$.

\begin{dfn}\label{holnomy} 
The holonomy group of an isotrivial abelian variety $\A$ 
is a group $G$ which satisfies the conditions a) and b) above 
and such that no quotient of $G$ satisfies them.
The holonomy map is the composition 
\begin{equation}\label{eqholnomy}
\pi_1(X-\Delta) \rightarrow G \rightarrow 
Aut A
\end{equation}
\end{dfn}

\bigskip 
Note that the first map in (\ref{eqholnomy}) can be described as
the homomorphism corresponding to the 
covering map $X'-\Delta' \rightarrow X-\Delta$, i.e. having the kernel 
isomorphic to $\pi_1(X'-\Delta')$.

In this paper we are concerned only with the case $X=\PP^2$. 
Since we assume in Th. \ref{main} b) that the image  
of the holonomy is non-trivial it follows that 
$\Delta$ has codimension one in $X$.



\subsection{Chow trace of isotrivial families}

Next recall Lang-Neron's finite generation result for abelian 
varieties over function fields starting with the 
definition of Chow trace (cf. \cite{langneron}, \cite{conrad}).
Given an extension $K/k$ of fields and an abelian variety $A$ over $K$, 
there exist an abelian variety 
$B$ over $k$ (called the \emph{Chow trace}) 
and homomorphism 
$\tau: B \otimes_k K \rightarrow A$ \footnote{$B \otimes_k K$ 
is the result of field extension of $B$.}
defined over $k$ 
such that for any extension $E/k$ disjoint from $K$, 
abelian variety $C$ over $E$ and morphism
$\alpha: C \rightarrow A$ over $KE$ there 
exists $\alpha': C \rightarrow B$ such that $\alpha=\tau \circ \alpha'$
(after appropriate field extensions of $A$ and $B$).
A description of the Chow trace in the case of relative Picard 
schemes is given in \cite{hindry} (cf. Prop.2.2, also \cite{shioda99}).
In the case of isotrivial abelian varieties, we have the following
(which in the case of relative Picard schemes is a consequence 
from \cite{hindry}):

\begin{prop} Let $\A \rightarrow X$ be an isotrivial 
abelian variety over $\CC(X)$ with holonomy $G$.
Then $\CC$-trace of $\A$ is isomorphic to the abelian subvariety $A^G$ of $A$ 
which is the maximal subvariety of $A$ fixed by the holonomy 
group $G$.
\end{prop}

\begin{proof} For any path $\gamma: [0,1] \rightarrow X-\Delta$ 
the identification 
(\ref{levin}) provides the map: $h_{\pi}: \pi^{-1}(0) \rightarrow 
\pi^{-1}(1)$ as the composition of a fixed identification of $\pi^{-1}(0)$ 
with the fiber over a point in $s^{-1}(\gamma(0))$ and a restriction 
of projection $(X'-\Delta') \times A \rightarrow (X'-\Delta')/G \subset
\A$ in 
(\ref{levin}) on the end point of the $s$-lift of path $\gamma$.
This is well-defined since the lift of $s$ is unique, but a change of 
the path $\gamma(t)$ results in a composition of $h_{\pi}$ 
with an automorphism from $G$. In particular, 
one has an identification of subvarieties $A^G$ of any two fibers
of $\pi$ and 
the map $A^G \times (X-\Delta) \rightarrow \A$ can 
be defined using continuation 
along paths. This yields the trace map  
$\tau: A^G\otimes k(X) \rightarrow \A$.

Next, given a map $T: B \times (X-A) \rightarrow \A$ commuting with 
projections on $X-A$, restricting it 
on a loop $\gamma$ 
in $X-\Delta$ yielding a holonomy transformation $g \in G$,
one sees that 
$T\vert_{(B \times \gamma(0))}: B \times \gamma(0) \rightarrow \A_{\gamma(0)}$ 
(the fiber of $\A$ over $\gamma(0)$)
has the image belonging to $A^G$ i.e. we have factorization 
of $T$ through $\tau$. This implies the universality 
property in the definition of trace.
\end{proof}

An automorphism group $G$ of a polarized abelian variety is finite
(cf. \cite{langeabvar} Ch.5 Cor.1.9) 
and (\ref{levin}) 
can be used to construct an isotrivial family of polarized abelian 
varieties for any etale covering of $X-\Delta$ 
with a Galois group $G \subset Aut A$.

With the notion of trace in place, one can state 
a function field version of the Mordell-Weil theorem as follows: 

\begin{theorem}\label{langneronref}(cf. \cite{langneron}) 
Let $K$ be a function field of a variety over a field $k$.
Let $A$ be an abelian variety
defined over $K$ and $\tau\colon  B \rightarrow \A$ is its trace. 
Then the  Mordell Weil group 
$A(K)/\tau B(k)$ is finitely generated.
\end{theorem}

\subsection{Examples of isotrivial families of 
abelian varieties with cyclic group of automorphisms.}\label{lefexample}

If the automorphism group of an abelian variety is cyclic, then 
{\it any} family of abelian varieties with such fiber 
has a cyclic holonomy group. 
Here is a way to obtain such examples. 
Jacobians of curves with cyclic automorphism groups have cyclic automorphism 
groups as well since by the 
Torelli theorem
$Aut(J(C))=Aut(C)/{\pm I}$ (resp. $Aut(J(C))=Aut(C)$)
for non-hyperelliptic curves (resp. for hyperelliptic curves).
(cf. \cite{langerec}, \cite{weil})).
As an example of curves with a cyclic automorphism group 
one can consider the curves $C_{p-2,p}$ with the following 
equation of the affine part (cf. \cite{lef},\cite{bertin},\cite{kalel} and section \ref{subsectiononstructure} below):
\begin{equation}\label{belyifamily}
    u^p=v^{p-2}(1-v)
\end{equation}

\begin{example}\label{examplelef} 
Let $\Phi(x,y)$ be a curve 
in $\CC^2$ which is the affine portion of a smooth projective 
curve having degree $p$ \footnote{Assumption of smoothness 
will be used below to show that the Mordell Weil rank in this 
case is zero. The construction in this example yields an isotrivial 
family of Jacobians for any $\Phi$.}. 
Consider the curve over $\CC(x,y)$ given by 
\begin{equation}\label{examplebelyi}
    u^p\Phi(x,y)=v^{p-2}(1-v)
\end{equation}
Over the complement to $\Phi(x,y)=0$ we have the family of curves 
isomorphic to the curve (\ref{belyifamily}). 
This family is trivialized over  $X'$ given by $z^p=\Phi(x,y)$.
The Jacobian of (\ref{examplebelyi}) over $\CC(x,y)$ provides
an example of an isotrivial abelian variety over this field.
\end{example}

\subsection{Abelian varieties of CM type}

A large class of examples of abelian varieties {\it admitting}
cyclic group of automorphisms, which will appear in several contexts below,
is given by abelian varieties of CM type.
Recall that a CM field is an imaginary quadratic 
extension of a totally real number field 
(cf. \cite{mum}, \cite{shimura}, \cite{milne}). 
A CM-algebra is a finite product of CM-fields. Such an algebra
$E$ is endowed with an automorphism $\iota_E$ such that 
for any $\rho: E \rightarrow \CC$ one has $ \rho \circ \iota_E=\bar \rho$
(the conjugation of $\rho$).
CM-type of a CM-algebra $E$ is:
\begin{equation}
\{\Phi \subset Hom(E,\CC) \vert Hom(E,\CC)=\Phi \cup \iota_E(\Phi), 
\Phi \cap \iota_E(\Phi)=\emptyset\}
\end{equation}

For a CM field $K$ of degree $g$ over $\QQ$, a CM type $\Phi$ is 
a collection of pairwise, not conjugate, embeddings 
$\sigma_1,...,\sigma_g$ of 
$K \rightarrow \CC$. 

For a CM algebra $E$ with a chosen CM-type $\Phi$ and 
a lattice in $E$ i.e. a subgroup $\Lambda$ such that 
$E=\Lambda \otimes_{\bf Z} \QQ$ (e.g. product of the rings 
of integers of each of CM-fields composing $E$)   
corresponds the torus $E \otimes \RR/\Lambda$
with the complex structure induced from 
the identification $E \otimes_{\QQ} \RR \rightarrow 
\CC^{{dim_{\QQ} E} \over 2}$ given by the direct sum of the 
homomorphisms $\phi \in \Phi$ where $\Phi$ is the CM type.
This complex torus is an abelian variety (cf. \cite{shimura},\cite{mum},
\cite{milne}).
The following example of a CM type appears below in the context 
of singularities:
\begin{example}\label{primitivity}
Let $p$ be an odd prime. Consider the set $\Phi$ of  
roots of unity of degree $p$ with positive imaginary part
(i.e. $exp ({{2 \pi i k} \over p})$ for $1 \le k \le {{p-1} \over 2}$).
The set of embeddings of $\QQ(\zeta_p)$ induced 
by the maps $exp({{2 \pi i} \over p}) \rightarrow \omega, \omega \in \Phi$ 
provides a CM type of $\QQ(\zeta_p)$.

Note that this CM type is primitive in the sense that the 
corresponding abelian variety is simple (cf. \cite{shimura}, section 8.4, 
p.64).

More generally, for a pair of primes $p,q$, the set of primitive roots 
of unity of degree $p \cdot q$ with a positive imaginary part provides
a CM type of the field $\QQ(\zeta_{pq})$. 

\end{example}

\subsection{Alexander polynomials}\label{sectionalexander}

The Alexander polynomial is an invariant of the fundamental 
group which allows one to state conditions under which 
the first Betti number of a cyclic cover is positive (cf. \cite{lectures}). 
Recall that for a group $G$ and a surjection $\sigma: G \rightarrow W$
onto a cyclic group $W$, 
one defines the Alexander polynomial as follows. 
Let $K=Ker \sigma$ and $K/K'$ be the abelianization of $K$.
It follows from the exact sequence
$$0 \rightarrow K'/K'' \rightarrow K/K'' \rightarrow W \rightarrow 0$$
 that 
$W$ acts on $K'/K''$ via conjugation on $K'$.
\begin{dfn} The Alexander polynomial $\Delta_G(t)$ of $G$ relative 
to the surjection $\sigma$ is the characteristic polynomial of a generator 
of $W$ acting on the vector space $K/K' \otimes \CC$ 
(this space has a finite dimension, cf. \cite{lectures};
in cases when $W$ is finite, one chooses the polynomial of the minimal 
degree among polynomials corresponding to different actions of generator).
Moreover, one has a cyclic decomposition $K'/K'' \otimes \CC=\oplus_i 
\CC[W]/\lambda_i$ \footnote{If $W$ is finite, then 
$\CC[W]$ is isomorphic to $\CC[t,t^{-1}]/(t^{\vert W \vert}-1)$ and 
$\lambda_i \in \CC[W]$ are viewed as polynomials in $\CC[t]$ 
having the minimal degree in its coset.
This definition is slightly 
different from the one used in \cite{duke} where only infinite $W$ was 
used. The reduction to the case where $W$ is infinite 
was done by replacing the projective curve $D$ by its affine 
portion such that the line at infinity $L$ is transversal to $D$.
If $D$ is irreducible then $H_1(\PP^2-D,\ZZ)=\ZZ_{deg D}$ but
$H_1(\PP^2-D\cup L,\ZZ)=\ZZ$. 
Moreover, for reduced $D$, 
the surjection $\pi_1(\PP^2-D\cup L) \rightarrow \ZZ$
given by the linking number with $D$ 
yields the same polynomial  
as surjection  $\pi_1(\PP^2-D) \rightarrow \ZZ_{deg D}$ given by the 
linking with $D$.}
in terms of which $\Delta_D(t)=\Pi_i \lambda_i(t)$ where
$t$ acts as a generator of $W$.
\end{dfn}
The properties of the Alexander polynomials  of the fundamental 
groups of the complements to the algebraic curves in $\PP^2$
are summarized by the following:

\begin{theorem}\label{summary}
Let $G=\pi_1(\PP^2-D)$ where $D$ is  
a projective curve of degree $d$ with arbitrary singularities 
and with $r$ irreducible components. Let $\Delta_D(t)$ denote 
the Alexander polynomial of $D$ relative to surjection $G \rightarrow 
\ZZ_{\deg D}=W$ sending a loop to  the class modulo $deg D$ of its 
total linking number with $D$
(cf. \cite{duke}) 


1. For each singularity $P$ of the curve $D$ 
denote by $\Delta_P(t)$  the Alexander polynomial of the local 
fundamental group 
$\pi_1(B_P-B_P \cap D)$ where $B_P$ is a small ball about $P$ in $\PP^2$ 
(as above, the Alexander polynomial is defined 
relative to surjection $\pi_1(B_P-B_P \cap D) \rightarrow \ZZ$ 
given by the total linking number with $B_P \cap D$).

Then the Alexander polynomial $\Delta_D(t)$
polynomial of $\pi_1(\PP^2-D)$ divides the product:
\begin{equation} \Pi_{P \in Sing(H \cap D)}\Delta_P(t)
\end{equation}
In particular the Alexander polynomial of $\pi_1(\PP^2-D)$ is  
cyclotomic.

2. Let $X_N$ be an $N$-fold cyclic branched covering space of $\PP^2$ 
ramified over $D$, and corresponding to a surjection of $W$ 
onto a cyclic group of order $N$. 
Then the characteristic polynomial of the generator of $W$ acting on 
$H_1(X_N,\CC)$   is equal to 
\begin{equation}
 \sum_i gcd(t^N-1,\lambda_i(t))
\end{equation}

3. With each singularity $P$ and a rational number $\kappa \in (0,1)$ 
one associates the ideal $I(P,\kappa)$ in the local ring of $P$
(the ideal of quasi-adjunction) \footnote{$I(P,\kappa)$ 
is defined in terms of the germ of the curve and $\kappa \in \QQ$
(cf. \cite{alexhodge});  
there is identification 
of the ideals of quasi-adjunction and the multiplier ideals (ibid.)}
with the following properties.
Let $\I_{\kappa} \subset \O_{\PP^2}$ be the ideal sheaf for which the support 
of $\O_{\PP^2}/\I_{\kappa}$ is the set of singularities of $D$ different than nodes
and stalk of $\I_{\kappa}$ at $P$ is $I(P,\kappa)$ then
\begin{equation}
  \Delta_D(t)=(t-1)^{r-1}\Pi_{\kappa} [(t-exp(-2 \pi i \kappa))(t-exp(2 \pi i \kappa))]^
{dim H^1(\PP^2,\I_{\kappa}(d-3-\kappa d))} 
\end{equation}
where the product is over all $\kappa={i \over d}, 1 \le i \le {d-1}$

In particular, for curves with singularities locally equivalent 
to $u^p=v^q$ only, one has: 
\begin{equation}
\Delta_D(t)=[{{(t^{pq}-1)(t-1)}\over {(t^p-1)}{(t^q-1)}}]^s
\end{equation}
where $s=dimH^1(\PP^2,\I(({1 \over p}+{1 \over q})d-3))$ and
 $\I$ is the ideal sheaf defined by conditions:

a) $\O_{\PP^2}/\I$ is supported
at singularities of $D$ 

b) stalk $\I$ at each singular point is 
the maximal ideal of the local ring. 
\footnote{Such a description of $\I$ is a consequence of a calculation yielding that 
the maximal ideal is the ideal of quasi-adjunction of $u^p=v^q$
and $\kappa=1-{1 \over p}-{1 \over q}$}
\end{theorem}
We refer to \cite{duke}, \cite{lectures} for proofs of these results but note 
that much of the reasoning in the proof of the  
theorem \ref{splittingresult}
is a Hodge theoretical refinement of the topological arguments in 
the proof of the first part of the theorem \ref{summary}. 
The local type of singularities 
of plane curves which come up in part 2 in theorem \ref{summary}
and associated mixed Hodge structures 
are discussed in the next section.
\end{section}

\section{Local Albanese varieties and singularities of CM-type}
\label{localalbsection}

The main result of this section is Theorem \ref{localalbanesemain},
describing the structure of the local Albanese varieties 
of unibranched singularities, showing 
that they have a CM type. Section \ref{mixedhodgeonlink}
contains a description 
of several constructions of the mixed Hodge structures associated
with plane curve singularities. The results mainly follow 
from previous discussions given in \cite{alexhodge} and \cite{jose}.
In section \ref{localalbaneseconstr} 
 we recall the definitions of the local Albanese variety 
following \cite{jose}. Then we introduce plane curve 
singularities of a CM-type as those for which the local 
Albanese varieties will have a CM type. The assertion that 
unibranched singularities have a CM type is proven  in section 
\ref{subsectiononstructure}.

\subsection{Mixed Hodge structures associated with a link.}\label{mixedhodgeonlink}

Let $f(x,y)$ be a germ of a plane curve singularity at the origin $(0,0)$.
Recall (cf.\cite{jose}) the comparison of
the limit of a mixed Hodge structure associated with 
degeneration $f(x,y)=t$ defined in the case of isolated 
singularities of arbitrary dimension in \cite{steenbrinknordic},
and the mixed Hodge structure on the 
cohomology of a punctured neighborhood of the exceptional 
set of a resolution of the singularity of $z^n=f(x,y)$ (or equivalently 
the link of the latter surface singularity) constructed in \cite{durfee}.

Let $V$ be a germ of an algebraic space having an isolated singularity at 
$P \in V \subset \CC^N$. Let $H^*_P(V)$ be the local cohomology of $V$. 
It is shown in \cite{steenbrinkarcata} 
(using a mapping cone construction) that $H^*_P(V)$  
supports a mixed Hodge structure.
The cohomology of the link $L$ of singularity of $V$, i.e. the intersection 
of $V$ with a small sphere in $\CC^N$ centered at $P$, 
is related to the local cohomology as follows:
\begin{equation}
H^*(L)=H^{*+1}_P(V)
\end{equation}
In particular the cohomology 
of $L$ receives a canonical mixed Hodge structure (cf. \cite{steenbrinkarcata}).
On the other hand, $L$ is a retract of a deleted 
neighborhood of the exceptional set of a resolution of singularity 
of $V$, which provides a description of this mixed Hodge structure 
using the presentation:  

\begin{equation}
\tilde V-E=\tilde V \bigcap \bar V-E
\end{equation}
where $\tilde V$ is a resolution of the germ $V$, 
$E$ is the exceptional set of the 
resolution, $\bar V$ is a smooth projective variety containing $\tilde V$.
Here one views $\tilde V$ as a small tubular neighborhood of the 
exceptional set $E$. In particular (cf. \cite{durfee}),
 one has the Mayer-Vietoris 
sequence, which is a sequence of the mixed Hodge structures:

\begin{equation}
    \longrightarrow H^k(\tilde V) \oplus H^k(\bar V-E) \rightarrow  
  H^k(\tilde V-E) \rightarrow H^{k+1}(\bar V) \longrightarrow 
 \end{equation}

The weights on $H^k(\tilde V)=H^k(E)$ (resp. $H^k(\bar V-E)$)
are $0,...,k$, since 
$E$ is a normal crossing divisor, (resp. $k,...,2k$  
since $\bar V-E$ 
is smooth). 
The weight of $H^{k+1}(\bar V)$ is $k+1$ since $\bar V$ is smooth 
projective. 
However the Gabber purity theorem yields that for $0 \le k < n$ the 
weights on $H^k(L)$ are less than or equal to $k$ and for $n \le k \le 2n-1$
are greater or equal than $k+1$ (cf. \cite{durfeehain}).

Applying this to the case when $V$ is the cyclic cover $V_{f,n}$ given by 
$z^n=f(x,y)$ one obtains for its link $L_{f,n}$ 
the mixed Hodge structure with weights
on $H^1(L_{f,n})$ being $0,1$ and weights on $H^2(L_{f,n})$ being $3,4$.

On the other hand, the vanishing cohomology of the family of germs 
$f(x,y)=t$, or equivalently the cohomology of Milnor fiber $M_f$,
supports the limit  mixed Hodge structure $H^1_{lim}(M_f)$
(with weights ($0,1,2$)). 
The following comparison between the mixed Hodge structures on $H^1(M_f)$ and 
$H^2(L_f)$ is given for example in \cite{jose}. 

\begin{prop}\label{comparemhs} Let $f(x,y)$ be a germ 
of a plane curve
(possibly reducible and non-reduced) 
with semi-simple monodromy of order $N$ 
and the Milnor fiber $M_f$. Let $L_{f,N}$ be link 
of the corresponding surface singularity $z^N=f(x,y)$.
Then there is the isomorphism of the mixed Hodge structures:
\begin{equation}\label{lemmaisomorphism}
Gr^W_3H^2(L_{f,N})(1)=Gr_1^WH^1(F_f)
\end{equation}
where the mixed Hodge structure on the left is the Tate twist of 
the mixed Hodge structure constructed in 
\cite{durfee} 
 and the one on the right is the mixed Hodge structure 
on vanishing cohomology constructed in \cite{steenbrinknordic}.
\end{prop}

From this Proposition we obtain the following:

\begin{corollary} If the monodromy of $f(x,y)=t$ 
is semisimple then the mixed Hodge structure on 
the Milnor fiber of $f(x,y)$ has type $(1,1),(1,0),(0,1)$.
The mixed Hodge structure on either 
side of (\ref{lemmaisomorphism})
is pure of weight $1$ and polarized. 
\end{corollary}

\begin{proof} 
Since the monodromy has a finite order, one has  $Gr^W_0=0$ 
(cf.  \cite{steenbrinknordic} p.547).
Moreover, $Gr^W_1H^1(F)=\oplus H^1(D_i)$ where $D_i$ are smooth curves
appearing in the semistable reduction of the family $f(x,y)=t$.
(ibid). Therefore we obtain 
the polarization of the term on the right hand side of 
(\ref{lemmaisomorphism}).

Note that the action of the monodromy on $Gr^W_2H^1(F_f)$ (resp.
$Gr^W_1H^1(F_f)$) is trivial (resp. does not have $1$ as an eigenvalue).

\end{proof}

\subsection{Local Albanese Variety}\label{localalbaneseconstr}

Given a pure Hodge structure $(H_{\ZZ},F)$ of weight $-1$,
one associates to it a complex torus as follows
(a more general case of mixed Hodge structures 
of type $(0,0),(0,-1),(-1,0),(-1,-1)$ is discussed in \cite{deligneIII}):

\begin{equation}\label{quotientformula}
     A_H=H_{\ZZ} \backslash H_{\CC}/F^0H_{\CC}
\end{equation}

In the case when the Hodge structure is polarized,
$A_H$ is an abelian variety.

\begin{dfn} Local Albanese variety $Alb_f$ of a plane curve singularity 
$f(x,y)=0$ is the abelian variety (\ref{quotientformula}) corresponding 
to the Hodge structure on homology $H_1(M_f,\ZZ)$ of the Milnor fiber 
which is dual to the cohomological mixed Hodge structure 
considered in the proposition \ref{comparemhs}.
\end{dfn}

\subsection{CM-singularities}\label{cmsingsection}
 
Recall that if the monodromy operator acting on the (co)homology of the
Milnor fiber is semisimple then it 
preserves the Hodge filtration (cf. \cite{steenbrinknordic}). 

\begin{dfn}\label{singcmtypedefinition}
 A plane curve 
singularity is called a singularity of CM type if its 
local Albanese variety is isogenous to a product of 
simple abelian varieties of CM type.
\end{dfn}

The local Albanese variety has the monodromy operator of the singularity 
as its automorphism. The following provides a description of the 
eigenvalues of the induced action on its tangent space at identity.

\begin{prop} Let $Alb_f$ be the local Albanese variety of singularity 
$f(x,y)=0$. The eigenvalues of the automorphism 
induced on the tangent space to $Alb_f$  
at identity by the monodromy operator of $f$ 
are the exponents $exp(2 \pi i \alpha)$ 
of the elements $\alpha \in \QQ$ of the spectrum of this singularity 
(cf. \cite{steenbrinknordic})
which satisfy $ 0 < \alpha <1$. 
\end{prop}

\begin{proof} Indeed the above tangent space can be identified with 
$Gr^0_FH^1(M_f)$ and the claim follows from the definition of the spectrum 
of singularity.
\end{proof}

For unibranched singularities of plane curves the spectrum 
was calculated in \cite{saito}. 

\begin{example}\label{pqsing}
For unibranched curve singularities with one characteristic pair, i.e. 
singularities with links equivalent to the links of singularity
$x^p=y^q$ where $gcd(p,q)=1$, 
the number of eigenvalues of the monodromy acting on 
$Gr^0_FH^1(M)$ ($M$ is the Milnor fiber)
is equal to ${{(p-1)(q-1)} \over 2}$. 
More precisely, the action on $H_1(M)$ is semi-simple and 
has as the characteristic polynomial (cf. e.g. \cite{Sumners})
\begin{equation}\label{cusppolyn}
\Delta_{p,q}={{(t^{pq}-1)(t-1)} \over {(t^p-1)(t^q-1)}}
\end{equation}
The characteristic polynomial of the action on $Gr^0_FH^1(M)$
is 

\begin{equation}
  \Pi (t-exp(-2 \pi \sqrt{-1} \alpha)),
\end{equation}
where
\begin{equation}\label{cmtype}  
\alpha={i \over p}+{j \over q} \ \ \ \ , 0 < \alpha <1, 0 <i<p, 0<j<q
\end{equation}
(cf. \cite{saito}  and references there).
In particular for $f(x,y)=x^2+y^3$ the only eigenvalue 
on $F^0$ is $exp({{2 \pi \sqrt{-1}} \over 6})$. More generally, 
for the singularity $x^2+y^p$ where $p$ is an odd prime, 
the field generated by the roots of (\ref{cusppolyn}) 
is $\QQ(\zeta_p)$ and the CM type corresponds to subset 
set $exp({{2 \pi \sqrt{-1} j}\over p})$ where ${1 \over 2}+{j \over p}<1$ 
i.e. coincides with the CM type discussed in example \ref{primitivity}. 
\end{example}

\begin{theorem}\label{cmsing}
Let $f(x,y)$ be a germ of a plane curve singularity 
such that the monodromy $T_f$ on $H^1(M_f,\CC)_{\ne 1}=
H^1(M_f)/Ker(T_f-Id)$
is semisimple. 
If the characteristic polynomial $\Delta_f(t)$ does not have multiple roots
different from 1 and $\Delta_f(-1) \ne 0$, 
then the singularity $f(x,y)$ has CM-type.
\end{theorem}

\begin{proof} Let $T_f$ denote linear operator induced by the monodromy $f$
of $Gr^W_1H^1(F_f)$.
Since the monodromy $T_f$ of the Milnor fiber is semisimple, 
as was mentioned earlier, it preserves the Hodge filtration, 
acts on $Alb_f$ and hence the algebra $End^{\circ}(Alb_f)=End(Alb_f) \otimes \QQ$
contains the algebra $\QQ[T_f]$. 
The latter has the dimension 
equal to the degree of the minimal polynomial of $T_f$ which 
is equal to $dim H^1(M_f,\CC)_{\ne 1}$
as follows from the assumption that the multiple eigenvalues
of the monodromy different from 1 are absent.

Next consider, the 
isogeny decomposition of $Alb_f$ into the product of abelian varieties 
$X_i$ on which appropriate powers $T_f^{d_i}$ have as eigenvalues only primitive 
roots of unity (cf. theorem 2.1 in \cite{lange}). Since $d_i \ge 3$ 
due to assumption $\Delta_f(-1) \ne 0$, each $X_i$ 
has sufficiently many complex multiplications, i.e. $End^{\circ} X_i$ contains 
semi-simple commutative algebra of rank $2dim X_i$. 
This subalgebra coincides with the center of $End^{\circ}(X_i)$. 
This implies the second inequality in:
\begin{equation}\label{albendineq}
      2 dim Alb_f =rk H_1(M_f,\CC)_{\ne 1} \le [End^{\circ}(Alb_f):\QQ]_{red}
\end{equation}
(notations as in \cite{milne} p.10 
\footnote{i.e. $[\Pi B_i:k]_{red}=\sum [B_i:k_i]^{1 \over 2}[k_i:k]$
for a product of simple algebras $B_i$ over $k$, 
with respective centers $k_i$.})
while the purity of Hodge structure 
on $H^1(M_f)_{\ne 1}$ implies the first equality.

Now, the Proposition 3.1 in \cite{milne} yields that in fact one has an 
equality 
in (\ref{albendineq}) 
and the claim follows (cf. for example def. 3.2 \cite{milne}).

Alternatively, one can use the Theorem 3.2 in \cite{lange} (cf. also \cite{agu})
to see that isogeny components $X_i$ have CM type 
(again using absence of multiple eignevalues of the monodromy 
acting on $H^1(M_f,\CC)_{\ne 1}$) 
and 
hence to obtain the conclusion of the Theorem \ref{cmsing}.
\end{proof}

\begin{example} Simple singularities have CM type (cf. \cite{jose}).
 Indeed, the characteristic
polynomials of the monodromy of Milnor fiber of simple singularities 
are readily available.
For singularity $A_{2k}$, it is equal to ${{(t^{2k}+1)(t-1)} \over {t^2-1}}$

For singularity $x^2y+y^{n-1}$ of type $D_n$ it is equal to 
\begin{equation}
\Delta(t)=(t^{n-1}+(-1)^{n-1})(t-1)
\end{equation}

For singularities $E_6,E_8$ i.e. $y^3+z^4,y^3+z^5$, 
the characteristic polynomials of monodromy are given in example \ref{pqsing}
and for $E_7$ i.e. $yz^3+y^3$ the characteristic polynomial is equal to $t^7-1$.
\end{example}

\begin{example} Consider singularity $f(x,y)=\Pi_{i=1}^{i=4}(x-\alpha_iy)=0$
where $\alpha_i$ are generic complex numbers. 
The characteristic
polynomial of the monodromy is $(t-1)^3(t^2+1)^2(t+1)^2$ and so Theorem
\ref{cmsing} cannot be applied. In fact the local Albanese coincides 
with the Jacobian of the only exceptional curve of the resolution 
of singularity $z^4=f(x,y)$. This curve is a 4-fold cyclic 
cover of $\PP^1$ totally ramified at 4 points. It cannot 
generically have CM type 
since the Jacobian of such curve surjects onto a 2-fold 
cover of $\PP^1$ branched at 4 points which hence represents a generic 
elliptic curve. 
\end{example}
\begin{example} Consider the singularity of plane curve with the 
Puiseux expansion:
\begin{equation}
x^{3 \over 2}+x^{{21} \over {10}}=x^{3 \over 2}+x^{{3 \over 2}+{{6} \over {2 \cdot 5}}}
\end{equation}
\end{example}
The Puiseux pairs are $(k_1,n_1)=(3,2),(k_2,n_2)=(6,5)$ 
which yields corresponding data $w_1=3, w_2=w_1n_1n_2+k_2=36$ (cf. \cite{saito})
and hence the 
characteristic polynomial of the monodromy of this singularity is
($\Delta_{p,q}$ is given by (\ref{cusppolyn}))
\begin{equation}
\Delta(t)=\Delta_{3,2}(t^5)\Delta_{36,5}(t)
\end{equation}
(cf. \cite{wall} for formulas for the characteristic polynomial 
of the monodromy in terms of Puiseux expansion), i.e. 
\begin{equation}\label{charpolmultiple}
\Delta(t)=[t^{10}-t^5+1][{{(t^{36 \cdot 5}-1)(t-1)} \over {(t^{36}-1)(t^5-1)}}]
\end{equation}
Since the cyclotomic polynomial of degree 10 divides the polynomials
in both brackets in (\ref{charpolmultiple}), $\Delta(t)$ 
has multiple roots. Nevertheless the local Albanese for this singularity 
is abelian variety of CM type (cf. Theorem \ref{localalbanesemain}).

\subsection{Structure of a local Albanese variety 
of singularities of CM type}\label{subsectiononstructure}

\begin{theorem}\label{productofjacobians}
Let $f(x,y)=0$ be a singularity with a semi-simple 
monodromy and let $N$ be the order of the monodromy operator.
The Albanese variety of germ $f(x,y)=0$ 
is isogenous to a product of Jacobians of the exceptional 
curves of positive genus for a resolution of: 
\begin{equation}\label{cycliccoversing}
z^N=f(x,y)
\end{equation}
\end{theorem}

\begin{proof} Denote by $P$ the isolated singularity of a germ 
$X$ of the surface (\ref{cycliccoversing})
and consider a resolution  $\tilde X \rightarrow X$ of $X$.
The dual graph of such a resolution does not contain cycles 
(cf. \cite{durfee75}) since the monodromy is assumed to be semi-simple
\footnote{Note that without the assumption that the surface singularity
has form (\ref{cycliccoversing}), the finiteness of the order 
of monodromy is not sufficient to conclude the absense of cycles
 (cf. \cite{artal}).}. Let $E=\cup E_i$ be the decomposition of the 
exceptional set of the resolution of (\ref{cycliccoversing}) 
into irreducible components.
We shall use the identification $H^2(L)=H^3_P(X)$ 
and the exact sequence (cf. \cite[Corollary (1.12)]{steenbrinkarcata}
of mixed Hodge structures on local cohomology:
\begin{equation}
  0 \rightarrow H^3_P(X) \rightarrow H^3_E(\tilde X) \rightarrow H^3(E)
\rightarrow 0
\end{equation}
The last term is trivial, i.e. one has the identification of the first two.
Moreover, one has the duality isomorphism (cf. \cite[(1.6)]{steenbrinkarcata}):
\begin{equation} 
H^3_E(\tilde X)=Hom(H^1(E),\QQ(-2))
\end{equation}
Since $H^1(E)=\oplus_i H^1(E_i)$,
we infer the isomorphism of Hodge structures:
\begin{equation}
   H^3(L)=\oplus Hom(H^1(E_i),\QQ(-2))
\end{equation}
The claim follows since for the curves $E_i$ 
having positive genus the Jacobians and  corresponding 
Albanese varieties are isomorphic.
\end{proof}

\begin{theorem}\label{localalbanesemain} 

Unibranched plane curve singularities have CM type. 

\end{theorem}

The proof of theorem \ref{localalbanesemain}
will consist of two steps. First, we shall 
show that for $f(x,y)$ unibranched 
all exceptional curves in a resolution of 
singularity (\ref{cycliccoversing}) 
are Belyi cyclic covers in the following sense:
\begin{dfn} A Belyi cyclic cover is a cyclic 
cover of $\PP^1$ branched at at most three points.
\end{dfn}
\noindent Secondly we shall use  the following
(cf. \cite{gross},\cite{koblitz}):
\begin{lemma}\label{belyijacobian} The Jacobian of a Belyi cyclic  cover is an 
abelian variety of CM type.
\end{lemma}
\noindent Then theorem \ref{localalbanesemain}
follows from the theorem \ref{productofjacobians}. A proof of lemma \ref{belyijacobian} is given in the Appendix for reader's convenience. 
\begin{proof} (of theorem \ref{localalbanesemain})
\begin{lemma}\label{belyitype}
 Exceptional curves of a resolution of the singularity 
(\ref{cycliccoversing}) are Belyi cyclic covers.
\end{lemma}

A resolution of the singularity  (\ref{cycliccoversing}) 
can be obtained as follows. Let $\pi: \tilde \CC^2 \rightarrow \CC^2$ 
be a sequence of blow ups of $\CC^2$ containing the germ $f(x,y)=0$
and yielding a resolution of the latter. Denote by $\tilde f: \tilde \CC^2 
\rightarrow \CC$ the composition of $\pi$ and $f: \CC^2 \rightarrow \CC$ 
and let $\lambda_N: \tilde \CC 
\rightarrow \CC$ be the $N$-fold cover of $\CC$ branched at the origin. 
Then one has the map of the normalization $\tilde \X$ of the fiber product 
of the maps $\tilde f$ and $\lambda_N$: 
\begin{equation}\label{fiberproduct}
\tilde \X \buildrel \N \over \rightarrow \tilde \CC^2 \times_{\CC} \tilde \CC \rightarrow X 
\end{equation}
Here $\tilde \X$ has at most simple surface singularities and 
their standard resolution, composed with the maps in 
(\ref{fiberproduct}), provides a resolution of $X$. Moreover, already 
$\tilde \X$ contains all the curves of a positive genus appearing in 
a resolution of $X$. 

Note that $\N$ replaces each exceptional curve $D$ of resolution $\tilde \CC^2 
\rightarrow \CC^2$ by its cyclic branched cover of degree $gcd(N,m)$ where 
$m$ is the multiplicity of $\pi^*(f)$ along $D$. Moreover, the ramification 
occurs at the intersection points of $D$ with the remaining exceptional curves.
To finish the proof of Lemma \ref{belyitype}, it is enough to show that 
each exceptional curve of $\pi$ has at most three 
intersections with remaining exceptional curves. This is the case as one 
can see, for example, from an inductive argument observing that the 
collection of exceptional curves on say $(k+1)$-th step 
in a resolution of $f=0$, is obtained from the 
collection of exceptional curves on step $k$ by blowing up 
up the intersection point of a proper preimage  
of $f$ appearing on the $k$-th step and the intersection 
point of exceptional curves of the $k$-th and $(k-1)$ steps.
Such triple intersection occurs iff the  exceptional curve
on $k-1$ step was tangent to the proper preimage of $f$ on that step. 
This yields the above claim on the number of 
intersections each $E$ can have with the
 remaining exceptional curves.

\end{proof}

We shall conclude this section indicating how one can  
obtain the identification of the CM type of isogeny components
of the local Albanese. The argument above implies that 
the components of the resolution of the surface 
singularity (\ref{cycliccoversing})
having non-trivial Jacobians 
(i.e. the components with a positive genus) correspond to the rapture 
points of the resolution tree of $f(x,y)=0$ (cf. \cite{wall})
\footnote{i.e. the point of the dual graph of resolution where 
with valency greater than 2.}. 
As follows from the discussion above, the valency of each rapture point
is equal to $3$.
The degree $d$ of the corresponding Belyi cover of a component $D$
of the exceptional set corresponding to such rapture point 
is equal to $gcd(N,m(D))$,
where $m(D)$ is the multiplicity of the pull back of the germ 
$f$ on the resolution on $D$. The ramification points of the Belyi cover 
correspond to the intersections with other exceptional curves in the 
resolution. The ramification index at the intersection of $D$ 
with another exceptional curve $D'$ is equal 
to ${m(D) \over {gcd(m(D),m(D')}}$.
This data consisiting of the degree and the ramification indices 
identifies the isomorphism type of the 
cyclic Belyi cover completely. Using 
the formulas in Lemma \ref{eigenmultiplicity} one can derive 
the CM type of corresponding Jacobian and hence the isogeny components
of local Albanese variety.

\begin{example} Consider the singularity $x^2+y^5$. The dual graph of its 
resolution has one rapture point. The multiplicity of the corresponding 
component is equal to $10$ with multiplicities of other three 
intersecting curves equal to $5,4,1$ respectively. The corresponding 
Belyi cover is 
\begin{equation}\label{25}
y^{10}=x^4(x-z)z^5
\end{equation}
(\ref{belyimultiplicity}) yields that the non-zero eigenvalues of the covering transformation
are $e^{{2 \pi \sqrt{-1}}\over {10}}$ and $e^{-{2 \pi \sqrt{-1}}\over {10}}$.
This determines the CM type of the Jacobian of the genus two curve (\ref{25})
corresponding to $\QQ(\zeta_{10})$.
\end{example}

\begin{example} For $y=x^{3 \over 2}+x^{7 \over 4}$
the characteristic polynomial of the monodromy is 
$$\Phi_{26}(t)\Phi_6(t^2)=\Phi_{26}(t)\Phi_{12}(t)$$
where $\Phi_n(t)$ denotes the cyclotomic polynomial of degree $n$.
The corresponding local Albanese variety is the 
product of simple CM-abelian varieties corresponding
to $\QQ(\zeta_{26})$ and $\QQ(\zeta_{12})$. The CM type of each field
is given by (\ref{cmtype}).  
\end{example}

\section{Splitting of Albanese varieties}\label{splittingsection}

In this section we show  
that the Albanese variety of certain cyclic branched 
covers of $\PP^2$ is isogenous to a product of abelian 
varieties of CM type. A similar result on the existence 
of an isogeny between the Albanese variety of a cyclic 
cover and a product of elliptic curves was obtained 
in \cite{jose}, but with much stronger restrictions on the singularities 
of the discriminant.

Recall that a construction of a model of 
cyclic branched cover with a given ramification curve can be given as
follows (cf. \cite{duke}). 
Let $D$ be a reduced irreducible curve in $\PP^2$ and let 
$\pi_1(\PP^2-D) \rightarrow \ZZ_N$ be a surjection onto a cyclic group.
The corresponding unramified cyclic covering of $\PP^2-D$ of degree $N$ is 
uniquely defined just by $D$, 
since the surjection $\pi_1(\PP^2-D) \rightarrow \ZZ_N$ coincides
(up to an automorphism of $\ZZ_N$) with the surjection 
given by the linking number with $D$. The  
 affine portion of the $N$ fold cyclic cover is given by 
\begin{equation}\label{equationNcover}
    z^N=F(x,y)
\end{equation}
where $F=0$ is an equation of $D$.
A compactification of the surface (\ref{equationNcover}),
combined with a resolution of singularities, yields a smooth 
model $X_N$ of covering space of $\PP^2$ branched 
over $D$. If $N=deg D$, then the projective closure of (\ref{equationNcover})
yields a model with isolated singularities in $\PP^3$. In 
the cases when $deg D>N$, a model with isolated singularities can be obtained
by the normalization of the projective closure.

\begin{theorem}\label{splittingresult}
  Let $D$ be a curve in $X=\PP^2$ with singularities of CM type only.
Then for a smooth projective 
model $X_N$ of  
$N$-fold cyclic cover of $\PP^2$ branched over $D$ 
(or equivalently the surface (\ref{equationNcover})), 
the Albanese variety 
$Alb(X_N)$ is isogeneous of a product of 
 abelian varieties of CM type.
\end{theorem}

\begin{proof} Let $\psi: X_N \rightarrow X=\PP^2$ be the 
projection of a smooth model of the $N$-fold cyclic 
cover (\ref{equationNcover}).
Let $\E=\cup E_i$ be the exceptional set. We shall denote by $\bar R$ 
the proper preimage of 
the branching locus of $\psi$ in $X_N$. This branching locus $R$
contains $D$ and possibly the line at infinity in $(x,y)$-plane 
of the cover (\ref{equationNcover}) 
(depending on the $gcd(deg D,N)$). 
The cohomology $H^1(X_N-\bar R)$ supports a mixed Hodge structure 
of type $(1,0),(0,1),(1,1)$ and hence one can consider the Albanese variety 
corresponding to the weight one part (cf. \cite{iitaka}, \cite{kike}).

{\it Step 1. Albanese of branched and unbranched covers.}
We claim that one has the identification:
\begin{equation}\label{identificationalbanese}
   Alb(X_N-\bar R \cup E_i)=Alb(X_N)
\end{equation}

We have the following exact sequence of the pair:
\begin{equation}\label{sequenceofpair}
\rightarrow 0=H^1(X_N,X_N-\bar R \cup E_i) \rightarrow  H^1(X_N) 
\rightarrow H^1(X_N-\bar R) \rightarrow 
\end{equation}
$$H^2(X_N,X_N-\bar R \cup E_i) 
\rightarrow H^2(X_N)
$$
The identification $H^i(X_N,X_N-\bar R \cup E_i)=
H_{4-i}(\bar R \cup E_i)$ shows that the left term is zero and 
that the right map is injective since the intersection form 
on $H^2(X_N)$ restricted on the subgroup generated by fundamental cycles 
of $\bar R,E_i$, is non-degenerate.

The sequence (\ref{sequenceofpair}) is a sequence of mixed Hodge structures
with the Hodge structure on $H^1(X_N-\bar R \cup E_i)$ having weights 1 and 2.
Hence (\ref{sequenceofpair}) induces the isomorphism
(\ref{identificationalbanese}).

{\it Step 2 Homology of unbranched cover and homology of cover 
of punctured regular 
neighborhood of branching locus} 

Let $U$ be a small regular neighborhood of $R$ in $\PP^2$. 
Since $R$ is ample, there exists a divisor $R' \subset U$ such that 
$\pi_1(R'-R \cap R') \rightarrow \pi_1(X-R)$ is a surjection
(by the Lefschetz hyperplane section theorem applied to quasi-projective 
manifold $X-R$). The latter surjection can be factored as
\begin{equation} 
 \pi_1(R'-R \cap R') \rightarrow \pi_1(U-R) 
\buildrel i_{U-R} \over \rightarrow \pi_1(X-R)
\end{equation}
and hence the right map is surjective.
If $K_{X-R} \subset \pi_1(X-R)$ (resp. $K_{U-R} \subset \pi_1(U-R)$)
is the kernel of surjection $lk_N: \pi_1(X-R) \rightarrow \ZZ_N$ 
(resp. the kernel of composition $lk_N \circ i_{U-R}$), then 
$i_{U-R}\vert_{K_{U-R}}: K_{U-R} \rightarrow K_{X-R}$ is surjective as well.
Hence denoting by $(U-R)_N$ the $N$-fold cover of $U-R$ corresponding 
to index $N$ subgroup $K_{U-R}$ on $\pi_1(U-D)$, we obtain the surjection:
\begin{equation}\label{mapuandxn}
        H_1((U-R)_N,\ZZ) \rightarrow H_1(X_N-\bar R\cup E_i) 
\end{equation}
(one verifies that the points at infinity do not provide
contributions since $D$ is always assumed to be transversal to 
the line at infinity cf. \cite{duke}).
Moreover, both groups support a mixed Hodge structure and 
hence the map (\ref{mapuandxn}) induced by embedding  induces 
a surjection of mixed Hodge structures.

{\it Step 3. Homology of punctured cover of 
regular neighborhood of branching locus and homology of 
links of singularities of cyclic cover.}

The covering space $(U-R)_N$ can be viewed as a regular 
neighborhood of the union of exceptional set $\E$ of $X_N$ 
for the map of $X_N$ onto the surface (\ref{equationNcover}) 
and the proper preimage of $R-Sing R$ in $X_N$
(where $Sing R$ is set of singular points of $R$). 
As such, it is a union 
of the $N$-fold cyclic covering $(U_{R-Sing R}^*)_N \subset X_N$ 
of a regular neighborhood $U_{R-Sing R}-(R-Sing R)$ 
of $R-Sing R \subset \PP^2$ with deleted $R-Sing R$ 
and punctured regular neighborhoods 
of the susbets $\E_P \subset \E$ of exceptional subset in
resolution $X_N$ of (\ref{equationNcover}) each being the 
preimage of the corresponding point $P \in Sing R$.

Let is consider the following part of the 
Mayer-Vietoris sequence, corresponding to just mentioned 
decomposition:
$$
(U-R)_N=\bigcup_P (U(\E_P)-\E_P) \cup (U_{R-Sing R}^*)_N
$$ 
(recall that $U(\E_P)-\E_P$ 
is a retract of the link $L_{N,P}$ of singularity of (\ref{equationNcover})
above $P$):
\begin{equation}\label{surjectionhomology}
  \oplus_{P \in Sing R} H_1(L_{N,P}) \oplus H_1((U_{R-Sing R}^*)_N)  
\rightarrow H_1((U-R)_N) \rightarrow 
\end{equation}
$$\rightarrow H_0((\bigcup_P L_{N,P})
\cap (U_{R-Sing R}^*)_N)
$$
Observe that the first homomorphism in (\ref{surjectionhomology}) is 
surjective.
Indeed the homomorphism following in the Mayer-Vietoris sequence  
the right map in (\ref{surjectionhomology}) is the map 
of a sum of zero dimensional
homology groups equivalent to injective map $\CC^{{\rm Card} Sing D}
\rightarrow \CC^{{\rm Card} Sing D+1}$ and hence is injective.

Finally note that the image of the map 
$H_1((U_{R-Sing R}^*)_N) \rightarrow H_1(X_N)$ 
is trivial. Indeed, 
the action of the covering group on $H_1((U_{R-Sing R}^*)_N)$ is trivial 
and hence the image of this group in $H_1(X_N)$ is trivial since the 
eigenspace on $H_1(X_N)$ corresponding to eigenvalue 1 has the same
rank as $H_1(X,\ZZ)$ and therefore is zero.

{\it Step 4. End of the proof}

Step 3 implies that composition map
$H_1(L_{N,P}) \rightarrow H_1((U-R)_N) \rightarrow H_1(X_N)$ is surjective (cf. \cite{duke}).
This yields the surjection of direct sum of the Albanese varieties 
corresponding to the remaining summands in the left term of 
(\ref{surjectionhomology}) (i.e. the local Albanese varieties
of all singular points of $D$) onto $Alb(X_N)$ and the claim of the theorem 
follows from Poincare complete reducibility theorem (cf. \cite{langeabvar}).

\end{proof}

\section{Proof of the main theorem and Examples}\label{proofmain}

In this section we shall finish the proofs of Theorem \ref{main}, 
the Corollary \ref{coromain},
and will discuss several examples.

\begin{proof} (of theorem \ref{main}).  Let $\pi_1(\PP^2-\Delta) \rightarrow 
\ZZ_d$ be the holonomy representation 
of the isotrivial family (\ref{morphismoffamily}).
Let $X_d$ denote (a smooth model of) the $d$-fold 
cyclic cover of $D$ branched over $\Delta$ and $\Delta' \subset X_d$
be such that $X_d-\Delta' \rightarrow \PP^2-\Delta$ is an unramified
cyclic cover. The holonomy group $\ZZ_d$ acts on $\A \times_{(\PP^2-\Delta)} X_d$
containing the $\ZZ_d$-invariant subset $(X_d-\Delta') \times A$ 
(with the diagonal action (cf. \ref{holonomysubsection})). 
Since by assumption the Chow trace of $\A$ is trivial, $MW(\A)$ 
is the group of section of morphism $\pi$ (cf. Theorem \ref{langneronref}).
We have the following:
\begin{prop} One has the canonical identification
\begin{equation} 
     MW(\A)=MW(X_d \times_{\PP^2} \A)^{\ZZ_d}=
Hom(Alb(X_d),A)^{\ZZ_d}
\end{equation}
\end{prop}

Indeed, assigning to $s: \PP^2-\Delta \rightarrow \A$ the 
regular section $(X_d-\Delta') \times_{(\PP^2-\Delta)} s(\PP^2-\Delta)$ 
of $(X_d-\Delta') \times_{\PP^2-\Delta} \A$ (which is invariant under the 
action of $\ZZ_d$) 
 provides the first isomorphism. 
The second follows from the identification:
\begin{equation}
MW(\A_d)=Mor(X_d,\A)=Hom(Alb(X_d),A)
\end{equation}

Since $Alb(X_d)$ is abelian variety of CM type 
it follows that the group 
$Hom(Alb(X_d),A)$
is trivial 
unless  the abelian variety 
$A$ is of CM type as well. Moreover, if $A$ is simple and corresponds 
to a cyclotomic
 field of degree $d$ then $rkMW$ is positive only if the decomposition 
of $Alb(X_d)$ into simple components contains $A$.  
This yields the inequality (\ref{inequalitymain}).
If $A$ is a component of $Alb(X_d)$ with multiplicity $s$ then 
$MW(A_d)=Hom(A^s,A)$ which has rank $s \cdot dim End^{\circ}(A)=s\phi(d)$.
\end{proof}

\begin{proof} (Of corollary \ref{coromain})
If none of characteristic polynomials of local monodromy of singularities
of $\Delta$  has a primitive $d$-th root of unity as a zero
then the global 
Alexander polynomial does not contain the factor 
$\Phi_d$ and hence the Albanese of $X_d$ cannot have as a factor 
a variety of CM type corresponding to the field $\QQ(\zeta_d)$. 
The example in the corollary (\ref{coromain}) discussed below.
\end{proof}

Part 3 of the Theorem \ref{main} provides an effective 
way to calculate the Mordell-Weil ranks of abelian 
varieties in the class described in its statement.
Below are several examples illustrating this procedure.

\begin{example}\label{finalexample} Consider the curve $C_{p,2}$ in $(u,v)$-plane over $\CC(x,y)$
given by
\begin{equation}\label{equationexample} 
u^p=v^2+(x^p+y^p)^2+(y^2+1)^p
\end{equation}
This curve over $\CC(x,y)$ is isotrivial since all curves 
\begin{equation}\label{examplecurve}
u^p=v^q+c, \ \ \ c \in \CC, \ \ \ c \ne 0
\end{equation}
are biholomorphic. Moreover (\ref{equationexample}) 
has as its discriminant the curve 
\begin{equation}\label{equationcurve}
  C_{p,2}: \ \ (x^p+y^p)^2+(y^2+1)^p
\end{equation}
The Alexander polynomial of the complement is the cyclotomic
polynomial of degree $2p$ (\cite{duke}): 
\begin{equation}
  \Phi_{2p}={{(t^{2p}-1)(t-1)}\over {(t^2-1)(t^p-1)}}
\end{equation}
The curve (\ref{examplecurve}) is the Belyi cyclic cover and
its Jacobian was described earlier as $A(\QQ(\zeta_{2p}),\Phi)$
with the CM type $\Phi$ as in example \ref{pqsing}. 
Moreover, the Albanese variety 
of the covering of degree $2p$ of $\PP^2$ ramified along 
$C_{p,2}$ is isomorphic to $A(\QQ(\zeta_{2p}),\Phi)$ as well.
Since $End^0(A(\QQ(\zeta_{2p}))=\QQ(\zeta_{2p})$ the claim follows.
Note that it follows that the above Jacobian is simple  
as a consequence of the discussion of example \ref{pqsing} 
since the CM type is primitive (cf. Example  \ref{primitivity}
or \cite{shimura}, section 8.4, 
p.64).
\end{example}

\begin{example} The Jacobian of the curve (\ref{examplebelyi})
considered in section
\ref{lefexample} has 
a Mordell-Weil rank equal to zero unless the Alexander polynomial 
of the curve $\Phi(x,y)$ has a root in $\QQ(\zeta_p)$. 
If $\Phi(x,y)$ is the equation 
(\ref{equationcurve}) then the curve (\ref{examplebelyi}) is birational 
over $\CC(x,y)$ to the curve (\ref{equationexample}) and hence
the Mordell-Weil rank for the Jacobian of (\ref{examplebelyi}) for 
such $\Phi(x,y)$ 
is $p-1$.
\end{example}

\begin{remark}
The Jacobian of the curve in example 
\ref{finalexample} is a simple isotrivial abelian variety over 
$\CC(x,y)$ such that rank of its Mordell-Weil group is equal to $p-1$.
In  particular the rank of abelian varieties over $\CC(x,y)$ can 
be arbitrarily large. 

In \cite{jose} it was shown that finding a bound 
on the rank of Mordell-Weil group for isotrivial elliptic curves
over $\CC(x,y)$ with discriminant having only nodes and cusps as its 
singularities is equivalent to finding a bound \footnote{either constant 
or depending on the degree of the discriminant}
 on the multiplicity 
of the factor $t^2-t+1$ in the Alexander polynomial of the discriminant
(in \cite{jose} more general cases including ADE singularities
were also considered).
For curves with nodes and cusps the largest known at the moment multiplicity 
is $4$ (cf. \cite{jose}).
Similarly, for abelian varieties $\A$ with generic fiber being 
a simple abelian 
variety of CM type corresponding to $\QQ(\zeta_d)$ and a CM type as 
in example \ref{pqsing} 
the rank of $MW(\A)$ is related to the multiplicity of the factor 
$\phi_d(t)$ (the cyclotomic polynomial of degree $d$) in the Alexander
polynomial of the discriminant. 
Note that there are very 
few known examples of plane curves with non-trivial Alexander polynomials 
and singularities beyond those of ADE type (cf. \cite{jose}, \cite{lectures}).  In particular, the largest multiplicity of $\phi_{pq}$ for $p,q>3$ 
is achieved for curves studied in \cite{josehirano}. They correspond
to threefolds given in the example below:
\end{remark}

\begin{example} 
Let 
\begin{equation}\label{hirano}
u^2=v^{2k+1}+(x^{2(2k+1)}+y^{2(2k+1)}+1-2x^{2k+1}+2(xy)^{2k+1}+y^{2k+1})
\end{equation} 
be the curve over $\CC(x,y)$. The discriminant is given by the 
second summand in the right hand side of (\ref{hirano}).
This is the curve studied in \cite{josehirano} where it was shown that 
the Alexander polynomial is 
\begin{equation} 
({{t^{2k+1}+1} \over {t+1}})^3
\end{equation}   
Generic fiber of fibration (\ref{hirano}) is 
hyperelliptic curve. For $2k+1=p$ its Jacobian 
is a simple abelian variety of a CM type and the  
rank of Mordell-Weil of the corresponding to (\ref{hirano}) 
family of Jacobians 
is $3(p-1)$.
\end{example}

\section{Appendix: Jacobians of Belyi covers}

In this appendix we shall prove the lemma \ref{belyijacobian} 
i.e. that the Jacobians of Belyi cyclic 
covers are abelian varieties of the CM type. 
Though the lemma \ref{belyijacobian} is apparently not new 
(cf. \cite{gross},\cite{koblitz})
the proof given below for convenience contains explicit 
formulas for the eigenvalues of the automorphisms of Belyi covers 
acting on the space of holomorphic 1-forms.

 \begin{proof}(of lemma \ref{belyijacobian}) We claim that a generator of the 
group of deck transformations of a cyclic 
Belyi cover acting on  $H_1$ does not have multiple eigenvalues.
Once this is established, an argument as in the proof of the 
theorem \ref{cmsing}, 
shows that for the Jacobian of such cover one has 
$2dim J=dim End^{\circ}(J)$ i.e. $J$ has a CM type.

Let $C \rightarrow \PP^1$ be a Belyi cyclic cover and 
let $d$ be its degree, i.e. the group of roots of unity of 
degree $\mu_d$ acts on $C$ with three points having non-trivial 
stabilizers. Let $a,b,c$ be the indices these stabilizers in $\mu_d$.
As a model of such Belyi cover (suitable for our calculations 
below, cf. proof of Lemma \ref{eigenmultiplicity}), 
one can choose the normalization of 
plane curve:
\begin{equation}\label{belyieq}
y^d=x^a(x-z)^bz^c , \ \ \ a+b+c=d.
\end{equation}  
The action of $\mu_d$ is given by 
$T: (x,y,z) \rightarrow (x,e^{2 \pi i \over d}y,z)$. 
Let $T_*$ be the induced map on $H_1(X,\CC)$. 
Now, the proper preimage for the map 
\begin{equation}
   (x,y) \rightarrow (x^d,y^bx^a)
\end{equation}
of the affine model of (\ref{belyieq}), i.e.
\begin{equation}
 y^d=x^a(x-1)^b
\end{equation}
yields a curve which has a component the Fermat curve 
\begin{equation}
 y^d=x^d-1.
\end{equation}
Since the Jacobian of Fermat curve is a product 
of abelian varieties of CM type (cf. \cite{koblitz})
this implies that the same is the case for 
cyclic Belyi covers.

\end{proof}

The following allows one effectively to calculate the CM type of 
local Albanese varieties in many cases. 
These formulas extend the special case presented in \cite{weil1}.
We have the following:
\begin{lemma}\label{eigenmultiplicity}
1. The multiplicity of the 
eigenvalue $\omega_d^j=e^{2 \pi \sqrt{-1} j\over d}$
 of $T_*$ acting on the space of holomorphic 1-forms of 
 Belyi cyclic cover as above is equal to:
\begin{equation}\label{belyimultiplicity}
-([-{{aj}\over d}]+[-{{bj}\over d}]+[{{(a+b)j} \over d}]+1)
\end{equation}
where $[ \cdot ]$ denotes the integer part.
In particular this multiplicity is equal either to zero or one.

2.Let $gcd(a,b,c,d)=1$ (i.e. the Belyi cover is irreducible). 
 Then the characteristic 
polynomial of the deck transformation acting on $H_1$ is 
given by 
\begin{equation} 
    \Delta(t)={{(t^{d}-1)(t-1)^2} \over{{(t^{gcd(a,d)}-1)}
{(t^{gcd(b,d)}-1)}{(t^{gcd(c,d)}-1)}}}
\end{equation}
\end{lemma}

\begin{proof}(of lemma \ref{eigenmultiplicity}) First note that 
the indices of stabilizers for the branching points of the 
cover (\ref{belyieq}) as the subgroups of 
 the covering group are $gcd(a,d),gcd(b,d),
gcd(c,d)$ respectively. Hence Riemann-Hurwitz formula yields that the genus 
of $C$ is given by (cf. \cite{kalel})
\begin{equation}\label{genus}
g={{d-gcd(a,d)-gcd(b,d)-gcd(c,d)+2} \over 2}.
\end{equation}
We shall represent explicitly the cohomology classes of $H^0(\Omega^1_C)$
and calculate the action of covering group on holomorphic 1-forms. 
Recall that the space of holomorphic 1-forms on a plane curve of degree $d$ 
can be identified with the space 
of adjoint curves of degree $d-3$, i.e. the curves of degree $d-3$ 
which equations at each singular point satisfy the adjunction 
conditions or equivalently belong to the adjoint 
ideal of this singularity. This can be made explicit
since any holomorphic 1-form can be written as the residue 
of 2-form on its complement, i.e. 
as
\begin{equation}\label{adjoint}
           {{ P(x,y)dx} \over {y^{d-1}}} \ \ \ \deg P \le d-3       
\end{equation}
The curve (\ref{belyieq}) may have singular 
points only at $(0,1),(1,0),(1,1)$ and near each the local equation 
is equivalent to $x^l+y^d=0$ (by abuse of language 
we shall refer to these points
as ``singular'' even if the curve is smooth there). 
To calculate the number of adjunction 
conditions we shall use the following (cf. \cite{merle}):
\begin{prop}
The conditions of adjunction for the singularity $y^d+x^l$ are 
the vanishing of the coefficients of monomials $x^ij^j$ such that 
$(i+1,j+1)$ is below or on the diagonal of the rectangle
with vertices $(0,0),(0,d),(l,0),(l,d)$. The number 
of adjunction conditions for singularity $y^d+x^l$ is equal to 
\begin{equation}
{{(d-1)(l-1)+gcd(d,l)-1} \over 2}
\end{equation}
\end{prop}
This implies that  the dimension of the space of curves of degree $d-3$ 
satisfying the conditions of adjunction at all three singularities
is greater or equal:
$${{(d-1)(d-2)} \over 2} -{{(d-1)(a-1)+gcd(d,a)-1} \over 2}-
{{(d-1)(b-1)+gcd(d,b)-1} \over 2}$$ 
$$-{{(d-1)(c-1)+gcd(d,c)-1} \over 2}=
{{d+2-gcd(a,d)-gcd(b,d)-gcd(c,d)}\over 2}.
$$
Comparison of this with the genus formula (\ref{genus})
shows that the conditions of adjunction imposed by three singular points 
are independent, i.e. 
one has the exact sequence
$$0 \rightarrow H^0(\Omega^1_{\tilde C}) \rightarrow 
H^0(\PP^2,\Omega^2_{\PP^2}(d)) \rightarrow \oplus_{s \in Sing C} M_s 
\rightarrow 0$$
where $\tilde C$ is the normalization of $C$ and $M_s$ is the quotient 
of the local ring of singular point by the adjoint ideal.

To calculate the action of $T_*$ on $H^1(\tilde C,\CC)$, we shall use 
the identification (\ref{adjoint})
of adjoints with the forms and that 
the action of $T^*$ on the monomials is given by 
$g(x^iy^j)=\omega_d^jx^iy^j$.  
Also note that the cardinality of the set 
of solutions to linear inequality
(for a fixed $j$) is given as follows:
\begin{equation}
Card \{ i \vert 0 < i, \ \ di+aj \le da\}=a+[-{{aj}\over d}]
\end{equation}

The multiplicity of the eigenvalue corresponding to the 
monomial $x^iy^{j-1}$
(i.e. $\omega_d^{j-1}$) in representation of $\mu_d$ in 
$H^0(\PP^2,\Omega^2_{\PP^2}(d))=H^0(\Omega^1_{\bar C})$ where $\bar C$ 
is a smoothing of $C$ is ${\rm Card} \{ i \vert 0 < i, i+j-1 \le d-3\}
=d-1-j$.
Hence the multiplicity of the eigenvalue $\omega^j$ is equal to 
\begin{equation} d-1-j-a-[-{{aj}\over d}]
-b-[-{{bj}\over d}]-c-[-{{cj}\over d}]=
\end{equation}
$$
=
-([-{{aj}\over d}]+[-{{bj}\over d}]+[{{(a+b)j} \over d}]+1).
$$
The last assertion of \ref{eigenmultiplicity}, i.e. that the multiplicity 
does not exceed 1, follows from the property $[x+y] \le [x]+[y]+1$. 

The formula for the characteristic polynomial can be derived 
using the additivity 
of zeta function similarly to the expression for the euler characteristic
obtained earlier
\end{proof}
To finish a proof of lemma \ref{belyijacobian}, just note that 
absence of multiple eigenvalues implies 
that the Jacobian must have a CM type (as in the proof of theorem
\ref{cmsing}).




\end{document}